\documentclass[A4wide,12pt]{article}

\usepackage{amsmath,amsthm,amsfonts,amssymb}
\usepackage{curves}

\newtheorem*{theorem}{Theorem}
\newtheorem{lemma}{Lemma}[section]

\newtheorem{corollary}[lemma]{Corollary}

\def\Z{\mathbb Z}
\def\adj{\!\sim\!}
\def\hs1{\hskip 1pt} 

\begin{document}

\title{${}$ \vskip -1.2cm Half-arc-transitive graphs of arbitrary even valency greater than 2}

\author{
Marston D.E. Conder 
\\[+1pt]
{\normalsize Department of Mathematics, University of Auckland,}\\[-4pt] 
{\normalsize Private Bag 92019, Auckland 1142, New Zealand} \\[-3pt]
{\normalsize m.conder@auckland.ac.nz}\\[+4pt]
\and
Arjana \v{Z}itnik
\\[+1pt]
{\normalsize Faculty of Mathematics and Physics, University of Ljubljana,} \\[-4pt] 
{\normalsize Jadranska 19, 1000 Ljubljana, Slovenia} \\[-3pt]
{\normalsize Arjana.Zitnik@fmf.uni-lj.si} 
}

\date{}

\maketitle

\begin{abstract}
A {\em half-arc-transitive\/} graph is a regular graph that is both vertex- and 
edge-transitive, but is not arc-transitive. 
If such a graph has finite valency, then its valency is even, and greater than $2$. 
In 1970, Bouwer proved that there exists a half-arc-transitive graph of every 
even valency greater than 2, by giving a construction for a family of graphs now 
known as $B(k,m,n)$, defined for every triple $(k,m,n)$ of integers greater than $1$ 
with $2^m \equiv 1 \mod n$.
In each case, $B(k,m,n)$ is a $2k$-valent vertex- and edge-transitive graph 
of order $mn^{k-1}$, and Bouwer showed that $B(k,6,9)$ is half-arc-transitive for all $k > 1$.  

For almost 45 years the question of exactly which of Bouwer's graphs are 
half-arc-transitive and which are arc-transitive has remained open, despite many 
attempts to answer it. 
In this paper, we use a cycle-counting argument to prove that almost all of the graphs 
constructed by Bouwer are half-arc-transitive.  In fact, we prove that $B(k,m,n)$ 
is arc-transitive only when $n = 3$, or $(k,n) = (2,5)$, 
or $(k,m,n) = (2,3,7)$ or $(2,6,7)$ or $(2,6,21)$. 
In particular, $B(k,m,n)$ is half-arc-transitive whenever $m > 6$ and $n > 5$.   
This gives an easy way to prove that there are infinitely many half-arc-transitive 
graphs of each even valency $2k > 2$. \\



\end{abstract}

\section{Introduction}
\label{intro} 

In the 1960s, W.T.~Tutte \cite{Tutte} proved that if a connected regular graph of odd valency is both 
vertex-transitive and edge-transitive, then it is also arc-transitive.  
At the same time, Tutte observed that it was not known whether the same was true for even valency. 
Shortly afterwards, I.Z.~Bouwer \cite{Bouwer} constructed a family of vertex- and edge-transitive 
graphs of any given even valency $2k > 2$, that are not arc-transitive.   

Any graph that is vertex- and edge-transitive but not arc-transitive is now known 
as a {\em half-arc-transitive\/} graph.  Every such graph has even valency, and 
since connected graphs of valency $2$ are cycles, which are arc-transitive, 
the valency must be at least $4$. 

Quite a lot is now known about  half-arc-transitive graphs, especially in the 
$4$-valent case --- see \cite{ConderMarusic,ConderPotocnikSparl,Marusic1998b,DM05} 
for example. 
Also a lot of attention has been paid recently to half-arc-transitive group actions 
on edge-transitive graphs --- see \cite {HujdurovicKutnarMarusic} for example. 
In contrast, however, relatively little is known about  half-arc-transitive graphs of higher valency.   
Bouwer's construction produced a vertex- and edge-transitive graph $B(k,m,n)$ 
of order $mn^{k-1}$ and valency $2k$ for every triple $(k,m,n)$ of integers 
greater than $1$ such that $2^{m} \equiv 1$ mod $n$, and Bouwer proved in \cite{Bouwer} 
that $B(k,6,9)$ is half-arc-transitive for every $k > 1$.   Bouwer also showed that the 
latter is not true for every triple $(k,m,n)$; for example, $B(2,3,7)$, $B(2,6,7)$ and $B(2,4,5)$ 
are arc-transitive. 

For the last $45$ years, the question of exactly which of Bouwer's graphs are 
half-arc-transitive and which are arc-transitive has remained open, despite a 
number of attempts to answer it. 
Three decades after Bouwer's paper, C.H.~Li and H.-S.~Sim \cite{LiSim} 
developed a quite different construction for a family of half-arc-transitive graphs, 
using Cayley graphs for metacyclic $p$-groups, 
and in doing this, they proved the existence of infinitely many half-arc-transitive graphs 
of each even valency $2k > 2$. 
Their approach, however, required a considerable amount of group-theoretic analysis. 

In this paper, we use a cycle-counting argument to prove that almost all of the graphs 
constructed by Bouwer in \cite{Bouwer} are half-arc-transitive, and thereby give an 
easier proof of the fact that there exist infinitely many half-arc-transitive graphs 
of each even valency $2k > 2$.  
Specifically, we prove the following: 

\begin{theorem} 
\label{thm:main} 
The 
graph $B(k,m,n)$ is arc-transitive only if and only if $n = 3$, or $(k,n) = (2,5)$, 
or $(k,m,n) = (2,3,7)$ or $(2,6,7)$ or $(2,6,21)$. 
In particular, $B(k,m,n)$ is half-arc-transitive whenever $m > 6$ and $n > 5$. 
\end{theorem} 

By considering the $6$-cycles containing a given $2$-arc, we prove in Section~\ref{sec:main} 
that $B(k,m,n)$ is half-arc-transitive whenever $m > 6$ and $n > 7$, 
and then we adapt this for the other half-arc-transitive cases in Section~\ref{sec:other}. 
In between, we prove arc-transitivity in the cases given in the above theorem 
in Section~\ref{sec:AT}.  
But first we give some additional background about the Bouwer graphs in 
the following section.

\section{Further background} 
\label{further} 

First we give the definition of the Bouwer graph $B(k,m,n)$, for every triple 
$(k,m,n)$ of integers greater than $1$ such that $2^{m} \equiv 1$ mod $n$. 

\smallskip
The vertices of $B(k,m,n)$ are the $k$-tuples $(a,b_2,b_3,\dots,b_k)$ 
with $a \in \Z_m$ and $b_j \in \Z_n$ for $2 \le j \le k$.  
We will sometimes write a given vertex as $(a,{\bf b})$, where ${\bf b} = (b_2,b_3,\dots,b_k)$. 
Any two such vertices are adjacent if they can be written as 
$(a,{\bf b})$ and $(a+1,{\bf c})$ 
where either ${\bf c} = {\bf b}$, 
or ${\bf c} = (c_2,c_3,\dots,c_k)$ differs from ${\bf b} = (b_2,b_3,\dots,b_k)$ in exactly 
one position, say the $(j\!-\!1)$st position, where $c_j = b_j +2^a$. 

\smallskip
Note that the condition $2^{m} \equiv 1$ mod $n$ ensures that $2$ is a unit mod $n$, 
and hence that $n$ is odd.  Also note that the graph is simple. 

\smallskip
In what follows, we let ${\bf e}_j$ be the element of $(\Z_n)^{k-1}$ with $j\hs1$th term 
equal to $1$ and all other terms equal to $0$, for $1 \le j < k$. 
With this notation, we see that the neighbours of a vertex $(a,{\bf b})$ 
are precisely the vertices $(a+1,{\bf b})$, $(a-1,{\bf b})$, $(a+1,{\bf b}+2^{a}{\bf e}_j)$ 
and $(a-1,{\bf b}-2^{a-1}{\bf e}_j)$ for $1 \le j < k$, and in particular, this shows that 
the graph $B(k,m,n)$ is regular of valency $2k$. 

\smallskip
Next, we recall that for all $(k,m,n)$, the graph $B(k,m,n)$ is 
both vertex- and edge-transitive; see \cite[Proposition~1]{Bouwer}.  
Also $B(k,m,n)$ is bipartite if and only if $m$ is even. 
Moreover, it is easy to see that $B(k,m,n)$ has the following three automorphisms:

\medskip \noindent (i)  $\theta$, of order $k-1$, taking 
each vertex $(a,{\bf b}) = (a,b_2,b_3,\dots,b_{k-1},b_k)$ 
to the vertex $(a,{\bf b}') = (a,b_3,b_4,\dots,b_k,b_2)$, obtained by shifting 
its last $k-1$ entries, 

\medskip \noindent (ii)  $\tau$, of order $m$, taking 
each vertex $(a,{\bf b}) = (a,b_2,b_3,\dots,b_{k-1},b_k)$ 
to the vertex $(a,{\bf b}'') = (a+1,2b_2,2b_3,\dots,2b_{k-1},2b_k)$, 
obtained by increasing its first entry $a$ by $1$ and multiplying the others by $2$, 
and 

\medskip \noindent (ii)  $\psi$, of order $2$, taking 
each vertex $(a,{\bf b}) = (a,b_2,b_3,\dots,b_{k-1},b_k)$ 
to the vertex $(a,{\bf b}'') = (a,2^{a}-1-(b_2+b_3+\dots+b_k),b_3,\dots,b_{k-1},b_k)$, 
obtained by replacing its second entry $b_2$ by $2^{a}-1-(b_2+b_3+\dots+b_k)$.

\medskip \noindent 
(Note: in the notation of  \cite{Bouwer}, the automorphism $\psi$ is $T_2 \circ S_2$.) 

\smallskip
The automorphisms $\theta$ and $\psi$ both fix the `zero' vertex $(0,{\bf 0}) = (0,0,\dots,0)$, 
and $\theta$ induces a permutation of its $2k$ neighbours that fixes each of the two 
vertices $(1,{\bf 0}) = (1,0,0,\dots,0)$ and $(-1,{\bf 0}) = (-1,0,0,\dots,0)$  
and induces two $(k\!-\!1)$-cycles on the others, while $\psi$ swaps 
$(1,{\bf 0})$ with $(1,{\bf e}_1) = (1,1,0,\dots,0)$,   
and swaps $(-1,{\bf 0})$ 
with $(-1,-2^{-1}{\bf e}_1) = (-1,-2^{-1},0,\dots,0)$, and fixes all the others. 

It follows that the subgroup generated by $\theta$ and $\psi$ fixes the 
vertex $(0,{\bf 0})$ and has two orbits of length $k$ on its neighbours, 
with one orbit consisting of the vertices of the form $(1,{\bf b})$ 
where ${\bf b} = {\bf 0}$ or ${\bf e}_j$ for some $j$, 
and the other consisting of those of the form $(-1,{\bf c})$ 
where ${\bf c} = {\bf 0}$ or $-2^{-1}{\bf e}_j$ for some $j$.  

By edge-transitivity, the graph $B(k,m,n)$ is arc-transitive if and only if it admits 
an automorphism that interchanges the `zero' vertex $(0,{\bf 0})$ with one of its 
neighbours, in which case the above two orbits of $\langle \theta, \psi \rangle$ 
on the neighbours of $(0,{\bf 0})$ are merged into one under the full automorphism 
group. 
We will use the automorphism $\tau$ in the next section. 

\smallskip
We will also use the following, which is valid in all cases, not just those with 
$m > 6$ and $n > 7$ considered in the next section.

\begin{lemma} 
\label{threearcs} 
Every $3$-arc $\,v_0 \adj v_1 \adj v_2 \adj v_3$ in $X$ with first vertex $v_0 = (0,{\bf 0})$ 
is of one of the following forms, with $r,s,t \in \{1,\dots,k\!-\!1\}$ in each case$\,:$ 
\\[+10pt]
\begin{tabular}{rl} 
\ {\rm (1)} & $(0,{\bf 0}) \adj (1,{\bf 0}) \adj (2,{\bf 0}) \adj (3,{\bf d})$,  
  where $\,{\bf d} = {\bf 0}$ or $\,4{\bf e}_r$,  \\[+4 pt]
\ {\rm (2)} & $(0,{\bf 0}) \adj (1,{\bf 0}) \adj (2,{\bf 0}) \adj (1,-2{\bf e}_r)$,  \\[+4 pt] 
\ {\rm (3)} & $(0,{\bf 0}) \adj (1,{\bf 0}) \adj (2,2{\bf e}_r) \adj (3,2{\bf e}_r\!+\!{\bf d})$,  
  where $\,{\bf d} = {\bf 0}$ or $\,4{\bf e}_s$,  \\[+4 pt]    
\ {\rm (4)} & $(0,{\bf 0}) \adj (1,{\bf 0}) \adj (2,2{\bf e}_r) \adj (1,2{\bf e}_r\!-\!{\bf d})$, 
  where $\,{\bf d} = {\bf 0}$ or $\,2{\bf e}_s$ with $s \ne r$,  \\[+4 pt] 
\ {\rm (5)} & $(0,{\bf 0}) \adj (1,{\bf 0}) \adj (0,-{\bf e}_r) \adj (1,-{\bf e}_r\!+\!{\bf d})$,  
  where $\,{\bf d} = {\bf 0}$ or $\,{\bf e}_s$ with $s \ne r$,  \\[+4 pt]  
\ {\rm (6)} & $(0,{\bf 0}) \adj (1,{\bf 0}) \adj (0,-{\bf e}_r) \adj (-1,-{\bf e}_r\!-\!{\bf d})$,  
  where $\,{\bf d} = {\bf 0}$ or $\,2^{-1}{\bf e}_s$,  \\[+4 pt] 
\ {\rm (7)} & $(0,{\bf 0}) \adj (1,{\bf e}_r) \adj (2,{\bf e}_r) \adj (3,{\bf e}_r\!+\!{\bf d})$,  
  where $\,{\bf d} = {\bf 0}$ or $\,4{\bf e}_s$,  \\[+4 pt] 
\ {\rm (8)} & $(0,{\bf 0}) \adj (1,{\bf e}_r) \adj (2,{\bf e}_r) \adj (1,{\bf e}_r\!-\!2{\bf e}_s)$,  \\[+4 pt] 
\end{tabular}
\begin{tabular}{rl} 
\ {\rm (9)} & $(0,{\bf 0}) \adj (1,{\bf e}_r) \adj (2,{\bf e}_r\!+\!2{\bf e}_s) \adj (3,{\bf e}_r\!+\!2{\bf e}_s\!+\!{\bf d})$,  
  where $\,{\bf d} = {\bf 0}$ or $\,4{\bf e}_t$,  \\[+4 pt]    
{\rm (10)} & $(0,{\bf 0}) \adj (1,{\bf e}_r) \adj (2,{\bf e}_r\!+\!2{\bf e}_s) \adj (1,{\bf e}_r\!+\!2{\bf e}_s\!-\!{\bf d})$,  
   \\ & \quad  where $\,{\bf d} = {\bf 0}$ or $\,2{\bf e}_t$ with $t \ne s$,  \\[+4 pt]    
{\rm (11)} & $(0,{\bf 0}) \adj (1,{\bf e}_r) \adj (0,{\bf e}_r) \adj (1,{\bf e}_r+{\bf e}_s)$,  \\[+4 pt]    
{\rm (12)} & $(0,{\bf 0}) \adj (1,{\bf e}_r) \adj (0,{\bf e}_r) \adj (-1,{\bf e}_r\!-\!{\bf d})$,  
  where $\,{\bf d} = {\bf 0}$ or $\,2^{-1}{\bf e}_s$,  \\[+4 pt]    
{\rm (13)} & $(0,{\bf 0}) \adj (1,{\bf e}_r) \adj (0,{\bf e}_r-{\bf e}_s) \adj (1,{\bf e}_r-{\bf e}_s\!+\!{\bf d})$,   
    where $s \ne r$, \\ & \quad  and $\,{\bf d} = {\bf 0}$ or $\,{\bf e}_t$ with $t \ne s$,  \\[+4 pt]    
{\rm (14)} & $(0,{\bf 0}) \adj (1,{\bf e}_r) \adj (0,{\bf e}_r-{\bf e}_s) \adj (-1,{\bf e}_r-{\bf e}_s\!-\!{\bf d})$,   
    where $s \ne r$, \\ & \quad  and $\,{\bf d} = {\bf 0}$ or $\,2^{-1}{\bf e}_t$,  \\[+4 pt]    
{\rm (15)} & $(0,{\bf 0}) \adj (-1,{\bf 0}) \adj (0,2^{-1}{\bf e}_r) \adj (1,2^{-1}{\bf e}_r\!+\!{\bf d})$,  
  where $\,{\bf d} = {\bf 0}$ or $\,{\bf e}_s$,  \\[+4 pt] 
{\rm (16)} & $(0,{\bf 0}) \adj (-1,{\bf 0}) \adj (0,2^{-1}{\bf e}_r) \adj (-1,2^{-1}{\bf e}_r\!-\!{\bf d})$,  
   \\ & \quad  where $\,{\bf d} = {\bf 0}$ or $\,2^{-1}{\bf e}_s$ with $s \ne r$,  \\[+4 pt] 
{\rm (17)} & $(0,{\bf 0}) \adj (-1,{\bf 0}) \adj (-2,{\bf 0}) \adj (-1,2^{-2}{\bf e}_r)$,  \\[+4 pt] 
{\rm (18)} & $(0,{\bf 0}) \adj (-1,{\bf 0}) \adj (-2,{\bf 0}) \adj (-3,-{\bf d})$,  
  where $\,{\bf d} = {\bf 0}$ or $\,2^{-3}{\bf e}_r$,  \\[+4 pt] 
{\rm (19)} & $(0,{\bf 0}) \adj (-1,{\bf 0}) \adj (-2,-2^{-2}{\bf e}_r) \adj (-1,-2^{-2}{\bf e}_r\!+\!{\bf d})$,  
   \\ & \quad  where $\,{\bf d} = {\bf 0}$ or $\,2^{-2}{\bf e}_s$ with $s \ne r$,  \\[+4 pt] 
{\rm (20)} & $(0,{\bf 0}) \adj (-1,{\bf 0}) \adj (-2,-2^{-2}{\bf e}_r) \adj (-3,-2^{-2}{\bf e}_r\!-\!{\bf d})$,  
   \\ & \quad  where $\,{\bf d} = {\bf 0}$ or $\,2^{-3}{\bf e}_s$,  \\[+4 pt] 
{\rm (21)} & $(0,{\bf 0}) \adj (-1,-2^{-1}{\bf e}_r) \adj (0,-2^{-1}{\bf e}_r) \adj (1,-2^{-1}{\bf e}_r\!+\!{\bf d})$, 
   \\ & \quad  where $\,{\bf d} = {\bf 0}$ or $\,{\bf e}_s$,  \\[+4 pt] 
{\rm (22)} & $(0,{\bf 0}) \adj (-1,-2^{-1}{\bf e}_r) \adj (0,-2^{-1}{\bf e}_r) \adj (-1,-2^{-1}{\bf e}_r\!-\!2^{-1}{\bf e}_s)$,  \\[+4 pt]  
{\rm (23)} & $(0,{\bf 0}) \adj (-1,-2^{-1}{\bf e}_r) \adj (0,-2^{-1}{\bf e}_r\!+\!2^{-1}{\bf e}_s) \adj (1,-2^{-1}{\bf e}_r\!+\!2^{-1}{\bf e}_s\!+\!{\bf d})$,  
   \\ & \quad  where $s \ne r$, and $\,{\bf d} = {\bf 0}$ or $\,{\bf e}_t$,  \\[+4 pt]    
{\rm (24)} & $(0,{\bf 0}) \adj (-1,-2^{-1}{\bf e}_r) \adj (0,-2^{-1}{\bf e}_r\!+\!2^{-1}{\bf e}_s) \adj (-1,-2^{-1}{\bf e}_r\!+\!2^{-1}{\bf e}_s\!-\!{\bf d})$, 
   \\ & \quad  where $s \ne r$, and $\,{\bf d} = {\bf 0}$ or $\,2^{-1}{\bf e}_t$ with $t \ne s$,  \\[+4 pt] 
{\rm (25)} & $(0,{\bf 0}) \adj (-1,-2^{-1}{\bf e}_r) \adj (-2,-2^{-1}{\bf e}_r) \adj (-1,-2^{-1}{\bf e}_r\!+\!2^{-2}{\bf e}_s)$,  \\[+4 pt] 
{\rm (26)} & $(0,{\bf 0}) \adj (-1,-2^{-1}{\bf e}_r) \adj (-2,-2^{-1}{\bf e}_r) \adj (-3,-2^{-1}{\bf e}_r\!-\!{\bf d})$, 
   \\ & \quad  where $\,{\bf d} = {\bf 0}$ or $\,2^{-3}{\bf e}_s$,  \\[+4 pt] 
{\rm (27)} & $(0,{\bf 0}) \adj (-1,-2^{-1}{\bf e}_r) \adj (-2,-2^{-1}{\bf e}_r\!-\!2^{-2}{\bf e}_s) \adj (-1,-2^{-1}{\bf e}_r\!-\!2^{-2}{\bf e}_s\!+\!{\bf d})$,  
   \\ & \quad  where $\,{\bf d} = {\bf 0}$ or $\,2^{-2}{\bf e}_t$ with $t \ne s$,  \\[+4 pt] 
{\rm (28)} & $(0,{\bf 0}) \adj (-1,-2^{-1}{\bf e}_r) \adj (-2,-2^{-1}{\bf e}_r\!-\!2^{-2}{\bf e}_s) \adj (-3,-2^{-1}{\bf e}_r\!-\!2^{-2}{\bf e}_s\!-\!{\bf d})$,  
   \\ & \quad  where $\,{\bf d} = {\bf 0}$ or $\,2^{-3}{\bf e}_t$.  \\ 
 \end{tabular}
\end{lemma}

\begin{proof}
This follows directly from the definition of $X = B(k,m,n)$. 
\end{proof} 

Note that the 28 cases given in Lemma~\ref{threearcs} fall naturally into $14$ pairs, 
with each pair determined by the form of the initial $2$-arc $\,v_0 \adj v_1 \adj v_2$. 

\smallskip
Also is it easy to see that 
the number of $3$-arcs in each case is \\[-4pt] 
$$\begin{cases}
\ k & \hbox{in cases 1 and 18,} \\[-1 pt]  
\ k\!-\!1 & \hbox{in cases 2 and 17,} \\[-1 pt] 
\ k(k\!-\!1) & \hbox{in cases 3, 6, 7, 12, 15, 20, 21 and 26,} \\[-1 pt] 
\ (k\!-\!1)^2 & \hbox{in cases 4, 5, 8, 11, 16, 19, 22 and 25,} \\[-1 pt]
\ k(k\!-\!1)^2 & \hbox{in cases 9 and 28,} \\[-1 pt]
\ (k\!-\!1)^3 & \hbox{in cases 10 and 27,} \\[-1 pt]
\ (k\!-\!1)^{2}(k\!-\!2) & \hbox{in cases 13 and 24,} \\[-1 pt]
\ k(k\!-\!1)(k\!-\!2) & \hbox{in cases 14 and 23,} 
\end{cases}
$$
and the total of all these numbers is $2k(2k\!-\!1)^2$, as expected.

\section{The main approach} 
\label{sec:main}

Let $k$ be any integer greater than $1$, and suppose that $m > 6$ and $n > 7$.   
We will prove that in every such case, the graph $X = B(k,m,n)$ is not arc-transitive, 
and is therefore half-arc-transitive. 

\smallskip
We do this simply by considering the ways in which a given $2$-arc or $3$-arc  
lies in a cycle of length $6$ in $X$. 
(For any positive integer $s$, an $s$-arc in a simple graph is a walk 
$\,v_0 \adj v_1 \adj v_2 \adj \dots \adj v_s$ of length $s$ in which any three consecutive vertices are distinct.)
By vertex-transitivity, we can consider what happens locally around the vertex $(0,{\bf 0})$. 

\begin{lemma} 
\label{girth} 
The girth of $X$ is $6$. 
\end{lemma}

\begin{proof}
First, $X = B(k,m,n)$ is simple, by definition. 
Also there are no cycles of length $3$, $4$ or $5$ in $X$, since in the list of cases 
for a $3$-arc $\,v_0 \adj v_1 \adj v_2 \adj v_3$ in $X$ with first vertex $v_0 = (0,{\bf 0})$  
given by Lemma~\ref{threearcs}, the vertex $v_3$ is never equal to $v_0$, 
the vertex $v_1$ is uniquely determined by $v_2$, 
and every possibility for $v_2$ is different from every possibility for $v_3$. 
On the other hand, there are certainly some cycles of length $6$ in $X$, such as 
$(0,{\bf 0}) \adj (1,{\bf 0}) \adj (2,2{\bf e}_k) \adj (1, 
 2{\bf e}_k) \adj (0,{\bf e}_k) \adj (1,{\bf e}_k) \adj (0,{\bf 0})$. 
\end{proof} 

Next, we can 
find all $6$-cycles 
based at the vertex $v_0 = (0,{\bf 0})$ in $X$. 

\begin{lemma} 
\label{sixcycles}
Up to reversal, every $6$-cycle based at the vertex $v_0 = (0,{\bf 0})$ 
has exactly one of the forms below, with $r$, $s,$ $t$ all different 
when they appear$\,:$ 
\\[+5pt]
$\bullet$ \ $(0,{\bf 0}) \adj (1,{\bf 0}) \adj (2,2{\bf e}_r) \adj (1,2{\bf e}_r) \adj (0,{\bf e}_r) \adj (1,{\bf e}_r) \adj (0,{\bf 0})$, \\[+3 pt] 
$\bullet$ \ $(0,{\bf 0}) \adj (1,{\bf 0}) \adj (0,-{\bf e}_r) \adj (1,-{\bf e}_r) \adj (2,{\bf e}_r) \adj (1,{\bf e}_r) \adj (0,{\bf 0})$, \\[+3 pt] 
$\bullet$ \ $(0,{\bf 0}) \adj (1,{\bf 0}) \adj (0,-{\bf e}_r) \adj (1,{\bf e}_s\!-\!{\bf e}_r) \adj (0,{\bf e}_s\!-\!{\bf e}_r) \adj (1,{\bf e}_s) \adj (0,{\bf 0})$,  \\[+3 pt] 
$\bullet$ \ $(0,{\bf 0}) \adj (1,{\bf 0}) \adj (0,-{\bf e}_r) \adj (-1,-{\bf e}_r) \adj (0,-2^{-1}{\bf e}_r) \adj (-1,-2^{-1}{\bf e}_r) \adj (0,{\bf 0})$, \\[+3 pt] 
$\bullet$ \  $(0,{\bf 0}) \adj (1,{\bf e}_r) \adj (2,2{\bf e}_s\!+\!{\bf e}_r) \adj (1,2{\bf e}_s\!-\!{\bf e}_r) \adj (0,{\bf e}_s\!-\!{\bf e}_r) \adj (1,{\bf e}_s) \adj (0,{\bf 0})$,  \\[+3 pt] 
$\bullet$ \ $(0,{\bf 0}) \adj (1,{\bf e}_r) \adj (0,{\bf e}_r) \adj (1,{\bf e}_s\!+\!{\bf e}_r) \adj (0,{\bf e}_s) \adj (1,{\bf e}_s) \adj (0,{\bf 0})$, \\[+3 pt] 
$\bullet$ \  $(0,{\bf 0}) \adj (1,{\bf e}_r) \adj (0,{\bf e}_r) \adj (-1,2^{-1}{\bf e}_r) \adj (0,2^{-1}{\bf e}_r) \adj (-1,{\bf 0}) \adj (0,{\bf 0})$,  \\[+3 pt] 
$\bullet$ \  $(0,{\bf 0}) \adj (1,{\bf e}_r) \adj (0,{\bf e}_r\!-\!{\bf e}_s) \adj (1,{\bf e}_r\!-\!{\bf e}_s\!+\!{\bf e}_t) \adj (0,{\bf e}_t\!-\!{\bf e}_s) \adj (1,{\bf e}_t) \adj (0,{\bf 0})$,  \\[+3 pt] 
$\bullet$ \  $(0,{\bf 0}) \adj (1,{\bf e}_r) \adj (0,{\bf e}_r\!-\!{\bf e}_s) \adj (-1,2^{-1}{\bf e}_r\!-\!{\bf e}_s) \adj (0,2^{-1}{\bf e}_r\!-\!2^{-1}{\bf e}_s)$ \\ ${}$ \qquad $\adj (-1,-\!2^{-1}{\bf e}_s) \adj (0,{\bf 0})$,  \\[+3 pt] 
$\bullet$ \  $(0,{\bf 0}) \adj (-1,{\bf 0}) \adj (0,2^{-1}{\bf e}_r) \adj (1,2^{-1}{\bf e}_r) \adj (0,-2^{-1}{\bf e}_r) \adj (-1,-2^{-1}{\bf e}_r) $ \\ ${}$ \qquad $\adj (0,{\bf 0})$,  \\[+3 pt] 
$\bullet$ \  $(0,{\bf 0}) \adj (-1,{\bf 0}) \adj (0,2^{-1}{\bf e}_r) \adj (-1,2^{-1}{\bf e}_r\!-\!2^{-1}{\bf e}_s) \adj (0,2^{-1}{\bf e}_r\!-\!2^{-1}{\bf e}_s)$ \\ ${}$ \qquad $\adj (-1,-\!2^{-1}{\bf e}_s) \adj (0,{\bf 0})$,  \\[+3 pt] 
$\bullet$ \  $(0,{\bf 0}) \adj (-1,{\bf 0}) \adj (-2,-2^{-2}{\bf e}_r) \adj (-1,-2^{-2}{\bf e}_r) \adj (-2,-2^{-1}{\bf e}_r) \adj (-1,-2^{-1}{\bf e}_r) $ \\ ${}$ \qquad $\adj (0,{\bf 0})$,  \\[+3 pt] 
$\bullet$ \  $(0,{\bf 0}) \adj (-1,-2^{-1}{\bf e}_r) \adj (0,-2^{-1}{\bf e}_r) \adj (-1,-2^{-1}{\bf e}_r\!-\!2^{-1}{\bf e}_s) \adj (0,-2^{-1}{\bf e}_s)$ \\ ${}$ \qquad $\adj (-1,-2^{-1}{\bf e}_s) \adj (0,{\bf 0})$,  \\[+3 pt] 
$\bullet$ \  $(0,{\bf 0}) \adj (-1,-2^{-1}{\bf e}_r) \adj (0,-2^{-1}{\bf e}_r\!+\!2^{-1}{\bf e}_s) \adj (1,2^{-1}{\bf e}_r\!+\!2^{-1}{\bf e}_s)$ \\ ${}$ \qquad $\adj (0,2^{-1}{\bf e}_r\!-\!2^{-1}{\bf e}_s) \adj (-1,-2^{-1}{\bf e}_s) \adj (0,{\bf 0})$,  \\[+3 pt] 
$\bullet$ \  $(0,{\bf 0}) \adj (-1,-2^{-1}{\bf e}_r) \adj (0,-2^{-1}{\bf e}_r\!+\!2^{-1}{\bf e}_s) \adj (-1,-2^{-1}{\bf e}_r\!+\!2^{-1}{\bf e}_s\!-\!2^{-1}{\bf e}_t)$ \\ ${}$ \qquad $\adj (0,2^{-1}{\bf e}_s\!-\!2^{-1}{\bf e}_t) \adj (-1,-2^{-1}{\bf e}_t) \adj (0,{\bf 0})$, or  \\[+3 pt] 
$\bullet$ \  $(0,{\bf 0}) \adj (-1,-2^{-1}{\bf e}_r) \adj (-2,-2^{-1}{\bf e}_r\!-\!2^{-2}{\bf e}_s) \adj (-1,-2^{-2}{\bf e}_r\!-\!2^{-2}{\bf e}_s)$ \\ ${}$ \qquad $\adj (-2,-2^{-2}{\bf e}_r\!-\!2^{-1}{\bf e}_s) \adj (-1,-2^{-1}{\bf e}_s) \adj (0,{\bf 0})$.  
\end{lemma} 

\begin{proof}
This is left as an exercise for the reader. 
It may be helpful to note that a $6$-cycle of the first form is obtainable as a $3$-arc of type 4 
with final vertex $2{\bf e}_r$ followed by the reverse of a $3$-arc of type 11 
with the same final vertex.  
The $6$-cycles of the other $15$ forms are similarly obtainable as the concatenation 
of a $3$-arc of type $i$  with the reverse of a $3$-arc of type $j$, 
for $(i,j) = (5,8)$, $(5,13)$, $(6,22)$, $(10,13)$, $(11,11)$, $(12,16)$, $(13,13)$, $(14,24)$, 
$(15,21)$, $(16,24)$, $(19,25)$, $(22,22)$, $(23,23)$, $(24,24)$ and $(27,27)$, respectively. 
Uniqueness follows from the assumptions about $m$ and $n$. 
\end{proof} 

\smallskip
\begin{corollary} 
\label{count6cycles} 
The number of different $6$-cycles in $X$ that contain a given $2$-arc 
$\,v_0 \adj v_1 \adj v_2$ with first vertex $v_0 = (0,{\bf 0})$ is always 
$0$, $1$ or $k$. 
More precisely, this number is 
\\[-12pt] 
\begin{center} 
\begin{tabular}{cl} 
${}$\hskip -18pt $0$ 
 & \hskip -8pt 
  for the $2$-arcs $(0,{\bf 0}) \adj (1,{\bf 0}) \adj (2,{\bf 0})$ 
  and  $(0,{\bf 0}) \adj (-1,{\bf 0}) \adj (-2,{\bf 0})$, \\ 
 & and those of the form $(0,{\bf 0}) \adj (1,{\bf e}_r) \adj (2,3{\bf e}_r)$,  \\
 &  and those of the form $(0,{\bf 0}) \adj (-1,-2^{-1}{\bf e}_r) \adj (-2,-(2^{-1}\!+\!2^{-2}){\bf e}_r)$, \\[+6pt]  
${}$\hskip -18pt $1$ 
 & \hskip -8pt 
  for the $2$-arcs of the form $(0,{\bf 0}) \adj (1,{\bf 0}) \adj (2,2{\bf e}_r)$,  \\
 &  and those of the form $(0,{\bf 0}) \adj (1,{\bf e}_r) \adj (2,{\bf e}_r)$,  \\
 & and those of the form $(0,{\bf 0}) \adj (1,{\bf e}_r) \adj (2,{\bf e}_r\!+\!2{\bf e}_s)$, when $k > 2$, \\
 & and those of the form $(0,{\bf 0}) \adj (-1,{\bf 0}) \adj (-2,-2^{-2}{\bf e}_r)$,  \\
 & and those of the form $(0,{\bf 0}) \adj (-1,-2^{-1}{\bf e}_r) \adj (-2,-2^{-1}{\bf e}_r)$, \\
 & and those of the form $(0,{\bf 0}) \adj (-1,-2^{-1}{\bf e}_r) \adj (-2,-2^{-1}{\bf e}_r\!-\!2^{-2}{\bf e}_s)$,  
  \\ & \quad when $k > 2$, \ and 
\end{tabular}
\begin{tabular}{cl} 
${}$\hskip -18pt $k$ 
 & \hskip -8pt 
  for the $2$-arcs of the form $(0,{\bf 0}) \adj (1,{\bf 0}) \adj (0,-{\bf e}_r)$,  \\
 &  and those of the form $(0,{\bf 0}) \adj (1,{\bf e}_r) \adj (0,{\bf e}_r)$,  \\
 & and those of the form $(0,{\bf 0}) \adj (1,{\bf e}_r) \adj (0,{\bf e}_r\!-\!{\bf e}_s)$,  when $k > 2$,  \\
 & and those of the form $(0,{\bf 0}) \adj (-1,{\bf 0}) \adj (0,2^{-1}{\bf e}_r)$,  \\
 & and those of the form $(0,{\bf 0}) \adj (-1,-2^{-1}{\bf e}_r) \adj (0,-2^{-1}{\bf e}_r)$, \\
 & and those of the form $(0,{\bf 0}) \adj (-1,-2^{-1}{\bf e}_r) \adj (0,-2^{-1}{\bf e}_r\!+\!2^{-1}{\bf e}_s)$,  
  \\ & \quad  when $k > 2$. \\[-3pt]  
\end{tabular}
\end{center}
In particular, every $2$-arc of the form $(0,{\bf 0}) \adj (1,{\bf 0}) \adj (2,2{\bf e}_r)$ 
lies in just one $6$-cycle, 
namely $(0,{\bf 0}) \adj (1,{\bf 0}) \adj (2,2{\bf e}_r) \adj (1,2{\bf e}_r) \adj (0,{\bf e}_r) \adj (1,{\bf e}_r) \adj (0,{\bf 0})$,
while every $2$-arc of the form  $(0,{\bf 0}) \adj (1,{\bf e}_r) \adj (0,{\bf e}_r)$ lies in $k$ $6$-cycles. 
\end{corollary}
 
\begin{proof}
This follows easily from inspection of the list of $6$-cycles given in Lemma~\ref{sixcycles},  
and their reverses. 
\end{proof}

At this stage, we could repeat the above calculations for $2$-arcs and $3$-arcs with first 
vertex $(1,{\bf 0}) = (1,0,0,\dots,0)$, but it is much easier to simply apply the 
automorphism $\tau$ defined in Section~\ref{further}.  

Hence in particular, the $2$-arcs $\,v_0 \adj v_1 \adj v_2$ with first 
vertex $v_0 = (1,{\bf 0})$ 
that lie 
in a unique $6$-cycle are those of the form $(1,{\bf 0}) \adj (2,{\bf 0}) \adj (3,4{\bf e}_r)$,  
or $(1,{\bf 0}) \adj (2,2{\bf e}_r) \adj (3,2{\bf e}_r)$, 
or $(1,{\bf 0}) \adj (2,2{\bf e}_r) \adj (3,2{\bf e}_r\!+\!4{\bf e}_s)$ when $k > 2$, 
or $(1,{\bf 0}) \adj (0,{\bf 0}) \adj (-1,-2^{-1}{\bf e}_r)$, 
or $(1,{\bf 0}) \adj (0,-{\bf e}_r) \adj (-1,-{\bf e}_r)$, 
or $(1,{\bf 0}) \adj (0,-{\bf e}_r) \adj (-1,-{\bf e}_r\!-\!2^{-1}{\bf e}_s)$ when $k > 2$. 

\medskip
Let $v$ and $w$ be the neighbouring vertices $(0,{\bf 0})$ and $(1,{\bf 0})$. 
We will show that there is no automorphism of $X$ that reverses the arc $(v,w)$. 

\medskip
Let $A = \{(2,2{\bf e}_r) : 1 \le r < k\}$, which is the set of all 
vertices $x$ in $X$ that extend the arc $(v,w)$ to a $2$-arc $(v,w,x)$ which lies in a unique $6$-cycle, 
and similarly, let $B = \{(0,-{\bf e}_r) : 1 \le r < k\}$, the set of all 
vertices that extend $(v,w)$ to a $2$-arc which lies in $k$ different $6$-cycles.
Also let $C = \{(-1,-2^{-1}{\bf e}_r) : 1 \le r < k\}$ and $D = \{(1,{\bf e}_r) : 1 \le r < k\}$  
be the analogous sets of vertices extending 
the arc $(w,v)$, as illustrated in Figure~\ref{ABCD}. 

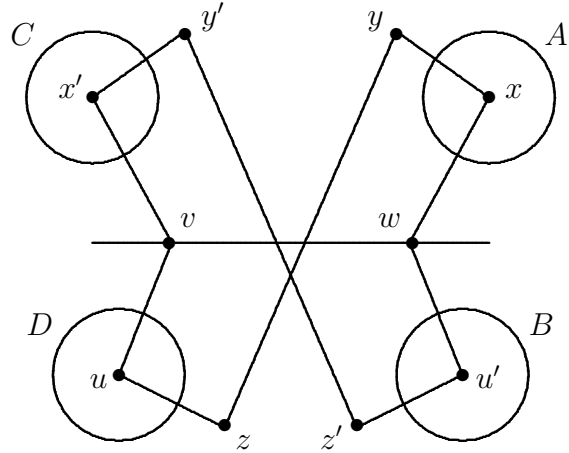
\begin{figure}[h] 
\begin{center} 
$$
\begin{picture}(240,170)(0,-50)
{\linethickness{0.8pt}
\put(0,0){\curve(50,30,200,30)}
\put(0,0){\curve(80,30,50,85)}
\put(0,0){\curve(170,30,200,85)}
\put(0,0){\curve(80,30,60,-20)}
\put(0,0){\curve(170,30,190,-20)}
\put(0,0){\curve(50,85,85,110)}
\put(0,0){\curve(200,85,165,110)}
\put(0,0){\curve(60,-20,100,-40)}
\put(0,0){\curve(190,-20,150,-40)}
\put(0,0){\curve(85,110,150,-40)}
\put(0,0){\curve(165,110,100,-40)}
\put(76,27){\makebox(0,0)[bl]{$\bullet$}}
\put(168,27){\makebox(0,0)[bl]{$\bullet$}}
\put(47,82){\makebox(0,0)[bl]{$\bullet$}}
\put(197,82){\makebox(0,0)[bl]{$\bullet$}}
\put(57,-23){\makebox(0,0)[bl]{$\bullet$}}
\put(187,-23){\makebox(0,0)[bl]{$\bullet$}}
\put(82,106){\makebox(0,0)[bl]{$\bullet$}}
\put(162,106){\makebox(0,0)[bl]{$\bullet$}}
\put(97,-42){\makebox(0,0)[bl]{$\bullet$}}
\put(147,-42){\makebox(0,0)[bl]{$\bullet$}}
\put(50,85){\bigcircle{50}}
\put(200,85){\bigcircle{50}}
\put(60,-20){\bigcircle{50}}
\put(190,-20){\bigcircle{50}}
\put(19,105){\makebox(0,0)[bl]{$C$}}
\put(221,105){\makebox(0,0)[bl]{$A$}}
\put(25,-5){\makebox(0,0)[bl]{$D$}}
\put(215,-5){\makebox(0,0)[bl]{$B$}}
\put(83,36){\makebox(0,0)[bl]{$v$}}
\put(158,36){\makebox(0,0)[bl]{$w$}}
\put(206,85){\makebox(0,0)[bl]{$x$}}
\put(37,85){\makebox(0,0)[bl]{$x'$}}
\put(154,110){\makebox(0,0)[bl]{$y$}}
\put(91,110){\makebox(0,0)[bl]{$y'$}}
\put(104,-48){\makebox(0,0)[bl]{$z$}}
\put(136,-48){\makebox(0,0)[bl]{$z'$}}
\put(49,-25){\makebox(0,0)[bl]{$u$}}
\put(195,-25){\makebox(0,0)[bl]{$u'$}}
}
\end{picture}
$$
\caption{\label{ABCD}$6$-cycles containing the edge from $v = (0,{\bf 0})$ to $w = (1,{\bf 0})$} 
\end{center} 
${}$\vskip -1cm 
\end{figure}

Now suppose there exists an automorphism $\xi$ of $X$ that interchanges the 
two vertices of the edge $\{v,w\}$.  
Then by considering the numbers of $6$-cycles that contain a given $2$-arc, 
we see that $\xi$ must interchange the sets $A$ and $C$, and interchange the 
sets $B$ and $D$.   

\smallskip
Next, observe that if $x = (2,2{\bf e}_r) \in A$, then the unique $6$-cycle containing 
the $2$-arc $v \adj w \adj x$ is $v \adj w \adj x \adj y \adj z \adj u \adj v$, 
where $(y,z,u) = ((1,2{\bf e}_r), (0,{\bf e}_r),(1,{\bf e}_r))$; 
in particular, the $6$th vertex $u = (1,{\bf e}_r)$ lies in $D$. 
Similarly, if $x' = (-1,-2^{-1}{\bf e}_r) \in C$, then the unique $6$-cycle containing 
the $2$-arc $w \adj v \adj x'$ is $w \adj v \adj x' \adj y' \adj z' \adj u' \adj w$, 
where $(y',z',u') = ((0,-2^{-1}{\bf e}_r), (-1,-{\bf e}_r), (0,-{\bf e}_r))$, 
and the $6$th vertex $u' = (0,-{\bf e}_r)$ lies in $B$. 

\smallskip
The arc-reversing automorphism $\xi$ must take every $6$-cycle 
of the first kind to a $6$-cycle 
of the second kind, and hence 
must take each $2$-arc of the form $v \adj u \adj z$ ($= (0,{\bf 0}) \adj (1,{\bf e}_r) \adj (0,{\bf e}_r)$)  
to a $2$-arc of the form $w \adj u' \adj z'$ ($= (1,{\bf 0}) \adj (0,-{\bf e}_r) \adj (-1,-{\bf e}_r)$).  
By Corollary~\ref{count6cycles}, however, each $2$-arc of the 
form $(0,{\bf 0}) \adj (1,{\bf e}_r) \adj (0,{\bf e}_r)$ lies in $k$ different $6$-cycles, 
while each $2$-arc of the form $(1,{\bf 0}) \adj (0,-{\bf e}_r) \adj (-1,-{\bf e}_r)$ 
is the $\tau$-image of the $2$-arc $(0,{\bf 0}) \adj (-1,-2^{-1}{\bf e}_r) \adj (-2,-2^{-1}{\bf e}_r)$ 
and hence lies in only one $6$-cycle. 
This is a contradiction, and shows that no such arc-reversing automorphism exists. 

\smallskip
Hence the Bouwer graph $B(k,m,n)$ is half-arc-transitive whenever $m > 6$ and $n > 7$. 

\bigskip
If $m \le 6$, then the order $m_2$ of $2$ as a unit mod $n$ is at most $6$, 
and so $n$ divides $2^{m_2} - 1 = 3, 7, 15, 31$ or $63$. 
In particular, if $m _2 = 2$, $3$, $4$ or $5$ then $n \in  \{3\}$, $\{7\}$, $\{5,15\}$ or $\{31\}$ 
respectively, while if $m_2 = 6$, then $n$ divides $63$ but not $3$ or $7$, 
so $n \in \{9, 21, 63\}$.
Conversely, if $n = 3, 5, 7, 9,$ $ 15, 21, 31$ or $63$, 
then $m_2 = 2$, $4$, $3$, $6$, $4$, $6$, $5$ or $6$, 
respectively, and of course in each case, $m$ is a multiple of $m_2$.  

\smallskip
We deal with these exceptional cases in the next two sections.

\section{Arc-transitive cases} 
\label{sec:AT} 

The following observations are very easy to verify. 
\medskip

When $n = 3$ (and $m$ is even), the Bouwer graph $B(k,m,n)$ is always arc-transitive,
for in that case there is an automorphism that takes each 
vertex $(a,{\bf b})$ to $(1\!-\!a,-{\bf b})$, and this 
reverses the arc from $(0,{\bf 0})$ to $(1,{\bf 0})$.

\medskip
Similarly, when $k = 2$ and $n = 5$ (and $m$ is divisible by $4$), 
the graph $B(k,m,n)$ is always arc-transitive, 
since it has an automorphism that takes  $(a,b_2)$ 
to $(1\!-\!a,-b_2)$ for all $a \equiv 0$ or $1$ mod $4$, 
and to $(1\!-\!a,2-b_2)$ for all $a \equiv 2$ or $3$ mod $4$,
and again this interchanges $(0,{\bf 0})$ with $(1,{\bf 0})$.

\medskip
Next, $B(2,m,7)$ is arc-transitive when $m = 3$ or $6$, 
because in each of those two cases there is an automorphism that takes
\\[-16pt] 
\begin{center}
\begin{tabular}{lll} 
$\quad(a,0)$ to $(1\!-\!a,0)$, \quad \ & 
$(a,1)$ to $(a\!+\!1,2)$,  \quad \   & 
$(a,2)$ to $(a\!-\!1,1)$, \\[+1 pt]  
$\quad (a,3)$ to $(\!-\!a,6)$, \quad \  &
$(a,4)$ to $(a\!+\!3,4)$, \quad \ & 
$(a,5)$ to $(1\!-\!a,5)$, \\[+1 pt] 
$\quad (a,6)$ to $(1\!-\!a,3)$, &  & \\[-6pt] 
\end{tabular} 
\end{center}
for every $a \in \Z_m$.  

\medskip
Similarly, $B(2,6,21)$ is arc-transitive since it has an automorphism taking 
\\[-16pt] 
\begin{center}
\begin{tabular}{lll} 
$\quad  (a,0)$ to $(1\!-\!a,0)$, \quad  \ & 
$(a,1)$ to $(a\!+\!1,2)$,  \quad \ & 
$(a,2)$ to $(a\!-\!1,1)$, \\[+1 pt]  
$\quad  (a,3)$ to $(5\!-\!a,6)$, \quad \ &
$(a,4)$ to $(a\!+\!3,11)$, \quad \ & 
$(a,5)$ to $(5\!-\!a,19)$, \\[+1 pt] 
$\quad  (a,6)$ to $(5\!-\!a,3)$, \quad &  
$(a,7)$ to $(1\!-\!a,14)$, \quad \ & 
$(a,8)$ to $(a\!+\!1,16)$, \\[+1 pt] 
$\quad  (a,9)$ to $(a\!-\!1,15)$, \quad &  
$(a,10)$ to $(5\!-\!a,20)$, \quad \ & 
$(a,11)$ to $(a\!+\!3,4)$, \\[+1 pt] 
$\quad  (a,12)$ to $(5\!-\!a,12)$, \quad &  
$(a,13)$ to $(5\!-\!a,17)$, \quad \ & 
$(a,14)$ to $(1\!-\!a,7)$, \\[+1 pt] 
$\quad  (a,15)$ to $(a\!+\!1,9)$, \quad &  
$(a,16)$ to $(a\!-\!1,8)$, \quad \ & 
$(a,17)$ to $(5\!-\!a,13)$, \\[+1 pt] 
$\quad  (a,18)$ to $(a\!+\!3,18)$, \quad &  
$(a,19)$ to $(5\!-\!a,5)$, \quad \ & 
$(a,20)$ to $(5\!-\!a,10)$, \\[-3pt] 
\end{tabular} 
\end{center}
for every $a \in \Z_m$.

\section{Other half-arc-transitive cases} 
\label{sec:other} 

In this final section, we give sketch proofs of half-arc-transitivity in all the remaining cases.
We first consider the three cases where $m$ and the multiplicative order of $2$ mod $n$ 
are both equal to $6$ (and $n = 9$, $21$ or $63$), 
and then deal with the cases where the multiplicative order of $2$ mod $n$ is less than $6$ 
(and $n = 5$, $7$, $15$ or $31$).  In the cases $n = 15$ and $n = 31$, 
we can assume that $m < 6$ as well. 

Note that one of these remaining cases is the one considered by Bouwer, 
namely where $(m,n) = (6,9)$.  As this is one of the exceptional cases, it is not representative 
of the generic case --- which may help explain why previous attempts to generalise 
Bouwer's approach did not get very far.

\smallskip
To reduce unnecessary repetition, we will introduce some further notation that will be used in most of these cases. 
As in Section~\ref{sec:main}, we let $v$ and $w$ be the vertex $(0,{\bf 0})$ and its 
neighbour $(1,{\bf 0})$.   

Next, for all $i \ge 0$, we define $V^{(i)}$ as the set of all 
vertices $x$ in $X$ for which $v \adj w \adj x $ is a $2$-arc that lies in exactly $i$ different $6$-cycles, 
and $W^{(i)}$ as the analogous set of vertices $x'$ in $X$ for which $w \adj v \adj x'$ 
is a $2$-arc that lies in exactly $i$ different $6$-cycles. 
(Hence, for example, the sets $A$, $B$, $C$ and $D$ used in Section~\ref{sec:main} 
are $V^{(1)}$, $V^{(k)}$, $W^{(1)}$ and $W^{(k)}$, respectively.) 

Also for  $i \ge 0$ and $j \ge 0$,  we define $T_{v}^{\,(i,j)}$ as 
the set of all $2$-arcs $v \adj u \adj z$ that come from a $6$-cycle of the
form $v \adj w \adj x \adj y \adj z \adj u \adj v$ with $x \in V^{(i)}$ and $u \in W^{(i)}$,  
and lie in exactly $j$ different $6$-cycles altogether, 
and define $T_{w}^{\,(i,j)}$ as the analogous set of all $2$-arcs $w \adj u' \adj z'$ 
that come from a $6$-cycle of the form $w \adj v \adj x' \adj y' \adj z' \adj u' \adj w$ 
with $x' \in W^{(i)}$ and $u' \in V^{(i)}$, and lie in exactly $j$ different $6$-cycles altogether.  

\smallskip
Note that if the  graph under consideration is arc-transitive, then it has an 
automorphism $\xi$ that reverses the arc $(v,w)$, and then clearly $\xi$ must 
interchange the two sets $V^{(i)}$ and $W^{(i)}$, for each $i$. 
Hence also $\xi$ interchanges the two sets 
$T_{v}^{\,(i,j)}$ and $T_{w}^{\,(i,j)}$ for all $i$ and $j$, 
and therefore $|T_{v}^{\,(i,j)}| = |T_{w}^{\,(i,j)}|$ for all $i$ and $j$. 
Equivalently,  if  $|T_{v}^{\,(i,j)}| \ne |T_{w}^{\,(i,j)}|$ for some pair $(i,j)$, 
then the graph cannot be arc-transitive, and therefore must be half-arc-transitive. 

\smallskip
The approach taken in Section~\ref{sec:main} was similar, but compared the 
$2$-arcs $v \adj u \adj z$ that come from a $6$-cycle of the
form $v \adj w \adj x \adj y \adj z \adj u \adj v$ where $x \in V^{(1)}$ and $u \in W^{(k)}$,  
with the $2$-arcs $w \adj u' \adj z'$ 
that come from a $6$-cycle of the form $w \adj v \adj x' \adj y' \adj z' \adj u' \adj w$ 
where $x' \in W^{(1)}$ and $u' \in V^{(k)}$.

\smallskip
We proceed by considering the first three cases below, in which the girth of the Bouwer 
graph $B(k,m,n)$ is $6$, but the numbers of $6$-cycles containing a given arc or $2$-arc are 
different from those found in Section~\ref{sec:main}.

\subsection{The graphs $B(k,6,9)$}
Suppose $m = 6$ and $n = 9$.  
This case was considered by Bouwer in \cite{Bouwer}, but it can also 
be dealt with in a similar way to the generic case in Section~\ref{sec:main}.  

Every $2$-arc lies in either $k$ or $k+1$ cycles of length $6$, 
and each arc lies in exactly $k$ distinct $2$-arcs of the first kind, 
and $k-1$ of the second kind. 
Next, the set $V^{(k)}$ consists of $(2,{\bf 0})$ and $(0,-{\bf e}_r)$ for $1 \le r < k$, 
while $W^{(k)}$ consists of $(-1,{\bf 0})$ and $(1,{\bf e}_s)$ for $1 \le s < k$. 

Now consider the $2$-arcs $v \adj u \adj z$ that come from a $6$-cycle of the
form $v \adj w \adj x \adj y \adj z \adj u \adj v$ with $x \in V^{(k)}$ and $u \in W^{(k)}$. 
There are $k^{2}-2k+2$ such $2$-arcs, and of these, $k\!-\!1$ lie in $k\!+\!1$ 
different $6$-cycles altogether  (namely the $2$-arcs of the form $(0,{\bf 0}) \adj (1,{\bf e}_s) \adj (2,{\bf e}_s)$), 
while the other $k^{2}-3k+3$ lie on only $k$ different 6-cycles.
In particular, $|T_{v}^{\,(k,k+1)}| = k\!-\!1$. 

On the other hand, among the $2$-arcs $w \adj u' \adj z'$ 
coming from a $6$-cycle of the form $w \adj v \adj x' \adj y' \adj z' \adj u' \adj w$ 
with $x' \in W^{(k)}$ and $u' \in V^{(k)}$, 
none lies in $k+1$ different 6-cycles.
Hence $|T_{w}^{\,(k,k+1)}| = 0$. 
Thus $|T_{v}^{\,(k,k+1)}|  \ne |T_{w}^{\,(k,k+1)}|$, and so the Bouwer graph $B(k,6,9)$ 
cannot be arc-transitive, and is therefore half-arc-transitive.

\subsection{The graphs $B(k,6,21)$ for $k > 2$}
Next, suppose $m = 6$ and $n = 21$, and $k > 2$. 
(The case $k = 2$ was dealt with in Section~\ref{sec:AT}.) 
Here every $2$-arc lies in one, two or $k$ cycles of length $6$, 
and each arc lies in exactly one $2$-arc of the first kind, 
and $k-1$ distinct $2$-arcs of each of the second and third kinds. 
The set $V^{(k)}$ consists of the $k\!-\!1$ vertices of the form $(0,-{\bf e}_r)$, 
while $W^{(k)}$ consists of the $k\!-\!1$ vertices of the form $(1,{\bf e}_s)$. 
(Note that this does not hold when $k = 2$.) 
Also $T_{v}^{\,(k,2)}$ consists of the $2$-arcs of the form $(0,{\bf 0}) \adj (1,{\bf e}_s) \adj (2,{\bf e}_s)$, 
so $|T_{v}^{\,(k,2)}| = k\!-\!1$,  
but on the other hand, $T_{w}^{\,(k,2)}$ is empty. 
Hence there can be no automorphism that reverses the arc $(v,w)$, 
and so the graph is half-arc-transitive.

\subsection{The graphs $B(k,6,63)$}
Suppose  $m = 6$ and $n = 63$. Then 
every $2$-arc lies in either one or $k$ different cycles of length $6$, 
and each arc lies in exactly $k$ $2$-arcs of the first kind, 
and $k\!-\!1$ of the second kind. 
The sets $V^{(k)}$ and $W^{(k)}$ are precisely as in the previous case, 
but for all $k \ge 2$, and in this case $T_{v}^{\,(k,1)}$ consists of the $2$-arcs of 
the form $(0,{\bf 0}) \adj (1,{\bf e}_s) \adj (2,{\bf e}_s)$, 
so $|T_{v}^{\,(k,1)}| = k\!-\!1$,  but on the other hand, $T_{w}^{\,(k,1)}$ is empty. 
Hence there can be no automorphism that reverses the arc $(v,w)$, 
and so the graph is half-arc-transitive.

\bigskip
Now we turn to the cases where $n = 5$, $7$, $15$ or $31$. 
In these cases, the order of $2$ as a unit mod $n$ is $4$, $3$, $4$ or $5$, 
respectively, and indeed when $n = 15$ or $31$ we may suppose 
that $m = 4$ or $m = 5$, while the cases $n = 5$ and $n = 7$ are much more tricky.

\subsection{The graphs $B(k,4,15)$}
Suppose $n = 15$ and $m = 4$.  
Then the girth is $4$, with $2k$ different $4$-cycles passing through the vertex $(0,{\bf 0})$. 
Apart from this difference, the approach taken in Section~\ref{sec:main} for counting $6$-cycles 
still works in the same way as before. 
Hence the graph $B(k,4,15)$ is half-arc-transitive 
for all $k \ge 2$.

\subsection{The graphs $B(k,5,31)$}
Similarly, when  $n = 31$ and $m = 5$, the girth 
is $5$, and the same approach as taken in Section~\ref{sec:main} using $6$-cycles 
works, to show that the graph $B(k,5,31)$ is half-arc-transitive 
for all $k \ge 2$.

\subsection{The graphs $B(k,m,5)$ for $k > 2$}
Suppose $n = 5$ and $k > 2$.  
(The case $k = 2$ was dealt with in Section~\ref{sec:AT}.) 
Here we have $m \equiv 0$ mod $4$, and the number of $6$-cycles is 
much larger than in the generic case considered in Section~\ref{sec:main} 
and the cases with $m = 6$ above, but a similar argument works. 

\smallskip
When $m > 4$, the girth of $B(k,m,n)$ is $6$, and every $2$-arc lies in 
either $2k$, $2k+3$ or $4k-4$ cycles of length 6. 
Also $V^{(2k+3)}$ consists of the $k\!-\!1$ vertices of the form $(0,-{\bf e}_r)$, 
while $W^{(2k+3)}$ consists of the $k\!-\!1$ vertices of the form $(1,{\bf e}_s)$,  
and then $T_{v}^{\,(2k+3,2k)}$ consists of the $2$-arcs of the form $(0,{\bf 0}) \adj (1,{\bf e}_s) \adj (2,{\bf e}_s)$, 
so $|T_{v}^{\,(2k+3,2k)}| = k\!-\!1$,  
but on the other hand, $T_{w}^{\,(2k+3,2k)}$ is empty. 

\smallskip
When $m = 4$, the girth of $B(k,m,n)$ is $4$, and the situation is similar, but slightly different. 
In this case, every $2$-arc lies in either $2k+2$, $2k+5$ or $6k-6$ cycles of length 6,  
and $|T_{v}^{\,(2k+5,2k+2)}| = k\!-\!1$ while $|T_{w}^{\,(2k+5,2k+2)}| = 0$. 

\smallskip 
Once again,  it follows that no automorphism can reverse the arc $(v,w)$, 
and so the graph is half-arc-transitive, for all $k > 2$ and all $m \equiv 0$ mod $4$.

\subsection{The graphs $B(k,m,7)$ for $(k,m) \ne (2,3)$ or $(2,6)$}
Suppose finally that $n = 7$, with $m \equiv 0$ mod $3$, but $(k,m) \ne (2,3)$ or $(2,6)$. 
Here $m \equiv 0$ mod $3$, and we treat four sub-cases separately: 
(a) $k = 2$ and $m > 6$; (b) $k > 2$ and $m = 3$; (c) $k > 2$ and $m = 3$;  
and (d)  $k > 2$ and $m > 6$. 

\medskip
In case (a), where $k = 2$ and $m > 6$, every $2$-arc lies in $1$, $2$ 
or $3$ cycles of length $6$.
Also the set $V^{(3)}$ consists of the single vertex $(0,-1)$, 
while $W^{(3)}$ consists of the single vertex $(1,1)$,  
and then $T_{v}^{\,(3,1)}$ consists of the single $2$-arc $(0,0) \adj (1,1) \adj (2,1)$, 
so $|T_{v}^{\,(3,1)}| = 1$,  
but on the other hand, $T_{w}^{\,(3,1)}$ is empty. 

\smallskip
In case (b), where $k > 2$ and $m = 3$, every $2$-arc lies in $0$, $2$ 
or $k$ cycles of length $6$. 
In this case the set $V^{(k)}$ consists of the $k\!-\!1$ vertices of the form $(0,-{\bf e}_r)$, 
while $W^{(k)}$ consists of the $k\!-\!1$ vertices of the form $(1,{\bf e}_s)$,  
and then $T_{v}^{\,(k,2)}$ consists of the $2$-arcs of the form $(0,{\bf 0}) \adj (1,{\bf e}_s) \adj (2,{\bf e}_s)$, 
so $|T_{v}^{\,(k,2)}| = k-1$,  
but on the other hand, $T_{w}^{\,(k,2)}$ is empty. 

\smallskip
In case (c), where $k > 2$ and $m = 6$, every $2$-arc lies in $3$, $k\!+\!1$ 
or $4k\!-\!3$ cycles of length $6$. 
Here $V^{(k+1)}$ consists of the $k\!-\!1$ vertices of the form $(0,-{\bf e}_r)$, 
while $W^{(k+1)}$ consists of the $k\!-\!1$ vertices of the form $(1,{\bf e}_s)$,  
and then $T_{v}^{\,(k+1,3)}$ consists of the $2$-arcs of the form $(0,{\bf 0}) \adj (1,{\bf e}_s) \adj (2,{\bf e}_s)$, 
and so $|T_{v}^{\,(k+1,3)}| = k-1$,  
but on the other hand, $T_{w}^{\,(k+1,3)}$ is empty. 

\smallskip
In case (d), where $k > 2$ and $m > 6$, every 2-arc lies in $1$, $k\!+\!1$ 
or $2k\!-\!2$ cycles of length $6$.
Next, if $k > 3$ then $V^{(k+1)}$ consists of the $k\!-\!1$ vertices 
of the form $(0,-{\bf e}_r)$, 
while $W^{(k+1)}$ consists of the $k\!-\!1$ vertices of the form $(1,{\bf e}_s)$,
but if $k = 3$ then $k+1 = 2k-2$, and $V^{(k+1)}$ contains also $(2,{\bf 0})$ 
while $W^{(k+1)}$ contains also $(-1,{\bf 0})$.   
Whether $k = 3$ or $k > 3$, the set $T_{v}^{\,(k+1,1)}$ consists of the $2$-arcs of the 
form $(0,{\bf 0}) \adj (1,{\bf e}_s) \adj (2,{\bf e}_s)$, and so $|T_{v}^{\,(k+1,1)}| = k-1$,  
but on the other hand, $T_{w}^{\,(k+1,1)}$ is empty. 

\smallskip 
Hence in all four of cases (a) to (d), no automorphism can reverse the arc $(v,w)$, 
and therefore the graph is half-arc-transitive. 

\smallskip
This completes the proof of our Theorem.  

\bigskip\bigskip

\noindent 
{\Large\bf Acknowledgements}
\bigskip

This work was undertaken while the second author was visiting the University 
of Auckland in 2014, and was partially supported by the ARRS (via grant P1-0294), 
the European Science Foundation (via EuroGIGA/GReGAS grant N1-0011), 
and the N.Z. Marsden Fund (via grant UOA1323).
The authors would also like to thank Toma{\v z} Pisanski and Primo{\v z} Poto\v{c}nik 
for discussions which encouraged them to work on this topic, 
and acknowledge considerable use of the {\sc Magma} system~\cite{Magma} for 
computational experiments and verification of some of the contents of the paper.



\begin{thebibliography}{99}

\bibitem{Magma}  
W. Bosma, J. Cannon and C. Playoust,
The {\sc Magma} Algebra System I: The User Language,
\emph{J. Symbolic Comput.} \textbf{24} (1997), 235--265.

\bibitem{Bouwer} 
I.Z. Bouwer,
Vertex and edge transitive, but not $1$-transitive, graphs, 
\emph{Canad. Math. Bull.} \textbf{13} (1970), 231--237.

\bibitem{ConderMarusic} 
M.D.E. Conder and D. Maru\v si\v c, 
A tetravalent half-arc-transitive graph with non-abelian vertex stabilizer, 
\emph{J. Combin. Theory Ser. B\/} \textbf{88} (2003) 67--76.

\bibitem{ConderPotocnikSparl} 
M.D.E. Conder, P. Poto\v{c}nik and P. \v{S}parl, 
Some recent discoveries about half-arc-transitive graphs, 
{\em Ars. Math. Contemp.} \textbf{8} (2015), 149--162. 
              
\bibitem{HujdurovicKutnarMarusic} 
A. Hujdurovi{\' c}, K. Kutnar and D. Maru{\v s}i{\v c}, 
Half-arc-transitive group actions with a small number of alternets, 
{\em J. Combin. Theory Ser. A\/} 124 (2014), 114--129. 
         
\bibitem{LiSim} 
C.H. Li and H.-S.~Sim,  
On half-transitive metacirculant graphs of prime-power order, 
\emph{J. Combin. Theory Ser. B}  \textbf{81} (2001), 45--57.

\bibitem{Marusic1998b} 
D. Maru\v si\v c, 
Recent developments in half-transitive graphs, 
\emph{Discrete Math.} \textbf{182} (1998), 219--231.

\bibitem{DM05} 
D. Maru\v si\v c, Quartic half-arc-transitive graphs with large vertex stabilizers, 
{\em Discrete Math.} {\bf 299} (2005), 180--193.

\bibitem{census}
P. Poto\v{c}nik, P. Spiga and G. Verret, 
A census of arc-transitive digraphs of valence $2$ and tetravalent 
graphs admitting a $\frac{1}{2}$-arc-transitive group action,
{\em Ars. Math. Contemp.} {\bf 8} (2015).

\bibitem{Tutte}
W.T. Tutte, 
\emph{Connectivity in graphs}, Univ. of Toronto Press, Toronto, 1966.


\end{thebibliography}
\end{document}